\documentclass[12pt,a4paper]{article}

\usepackage{amssymb, amsmath, amsthm, amsfonts}

\usepackage[cm]{fullpage}
\usepackage[english]{babel}
\usepackage[T1]{fontenc}
\usepackage[pdftex]{graphicx}
\usepackage{booktabs}
\usepackage{xcolor}
\usepackage{hyperref}
\hypersetup{
    colorlinks = false
}

\usepackage{mathtools}
\usepackage{graphicx}
\usepackage{epstopdf}
\usepackage{dsfont}
\usepackage{mathrsfs}
\usepackage{physics}
\usepackage{bm}

\newtheorem{thm}{Theorem}
\newtheorem{lem}{Lemma}

\newtheorem{cor}{Corollary} 

\newcommand{\R}{\mathbb{R}}

\newcommand{\ph}{\varphi}

\usepackage{csquotes}
\usepackage[
backend=biber,
style=phys,
sorting=none
]{biblatex}
\title{Fully discrete Galerkin scheme for a semilinear subdiffusion equation with nonsmooth data and time-dependent coefficient}
\author{\L ukasz P\l ociniczak\thanks{Faculty of Pure and Applied Mathematics, Wroc{\l}aw University of Science and Technology, Wyb. Wyspia{\'n}skiego 27, 50-370 Wroc{\l}aw, Poland, email: lukasz.plociniczak@pwr.edu.pl}
\and 
Kacper Ta\'zbierski$^{*,}$\thanks{Email: kacper.tazbierski@pwr.edu.pl}}

\begin{document}

\maketitle

\begin{abstract}
We couple the L1 discretization of the Caputo fractional derivative in time with the Galerkin scheme to devise a linear numerical method for the semilinear subdiffusion equation. Two important points that we make are: nonsmooth initial data and time-dependent diffusion coefficient. We prove the stability and convergence of the method under weak assumptions concerning regularity of the diffusivity. We find optimal pointwise in space and global in time errors, which are verified with several numerical experiments.
\end{abstract}

\textbf{Keywords:} L1 method, time-dependent coefficient, subdiffusion, nonsmooth data, Galerkin method\\

\section{Introduction}
We consider the following time-fractional semilinear diffusion problem
\begin{equation}\label{eq:maineq}
\begin{cases}
     \partial_t^\alpha u=\nabla \cdot \left( D(x,t)\nabla u\right)+f(x,t,u), & x\in \text{int }\Omega\subset\R^d, \; t\in(0,T] \\
     u(x,t) = 0, & x\in\partial\Omega,t\in(0,T],\\
     u(x,0) = \ph(x), &  x\in \Omega,
\end{cases}
\end{equation}
with $\alpha\in(0,1)$ and $\ph(x)\in H^p_0$, $p\in(0,2]$ where $H^p_0$ is the (fractional) Sobolev space of functions with vanishing trace. $\partial^\alpha_t$ is the Caputo derivative defined for $\alpha\in(0,1)$ with the following formula
\begin{equation}
    \partial_t^\alpha y(t)=\frac{1}{\Gamma(1-\alpha)}\int\limits_0^t(t-s)^{-\alpha}y_\tau(x,\tau)d\tau.
\end{equation}
We will also write $\partial^m_t$ with $m\in\mathbb{N}$ to denote the classical partial derivative with respect to time. 

The main equation \eqref{eq:maineq} arises as a model of various diffusion phenomena that exhibit memory effects. For example, a randomly walking particle for which the mean squared displacement increases as a power function of time is said to undergo an anomalous diffusion \cite{metzler2000random}. For the subdiffusive regime, that is, when the exponent of this power-type relation is less than one, the Caputo derivative naturally emerges as a time evolution operator. In the deterministic counterpart, one can also show that the time-fractional operator arises when studying the waiting-time phenomenon in fluid dynamics. In the hydrological setting, this produces a nonlocal conservation law, which in turn leads to \eqref{eq:maineq} (see \cite{plociniczak2015analytical}). There are many other occurrences of anomalous diffusion, which are usually modeled with fractional operators (both in space and time). For example, experiments indicate anomalous evolution in protozoa migration \cite{alves2016transient}, single particle tracking in biophysics \cite{Sun17, wong2004anomalous}, plasma physics \cite{Del05}, astrophysics \cite{lawrence1993anomalous}, and viscoelasticity \cite{bagley1983theoretical} to list only a few items. 

The main purpose of this paper is to devise a fully discrete numerical method for solving \eqref{eq:maineq} and to prove its stability and convergence under the assumptions of nonsmooth initial data. This is a continuation of our previous results \cite{Plociniczak2022} where we have investigated the semi-discrete scheme. Here, we couple it with the L1 discretization of the Caputo derivative. This approximation for the Caputo derivative is based on a piecewise linear basis and was introduced in \cite{oldham1974fractional} to compute the Riemann-Liouville derivative. Further analyses of this method can be found, for example, in \cite{Plociniczak2023} where the optimal error constant for smooth functions was found. For such a regularity, the L1 method achieves a global in time order $2-\alpha$ that can deteriorate to mere $\alpha$ for H\"older functions. The latter result was proved in \cite{Stynes2016,kopteva2019error} assuming the existence of second derivatives (but not necessarily bounded; we will discuss this matter below). The L1 method has been applied to many various problems and equations, and we would like to mention only the most relevant ones for our work. In particular, Galerkin discretization in space has been coupled with the L1 method to solve the subdiffusion equarion with constant diffusivity and nonsmooth data in \cite{jin2019subdiffusion} where the authors used operator theory to find optimal errors of the method. This result has been generalized for a space-dependent diffusivity, for example, in \cite{kopteva2019error} by a completely different technique that unifies many methods into one framework. A semilinear subdiffusion equation has been thoroughly analyzed in \cite{jin2018numerical}. A complete review of the L1 method can be found in \cite{stynes2022survey}. As for both space- and time-dependent coefficients, to the best of our knowledge, one can find only a limited number of relevant works. For example, in \cite{jin2019subdiffusion} the authors used a variant of the perturbation method to find the optimal error bound for the discrete in space Galerkin finite element method (this result has been further improved in \cite{jin2020subdiffusion}). A completely different approach was taken in \cite{mustapha2018fem} where an energy-type approach was developed to tackle the time-dependent diffusivity case which causes severe difficulties. This approach was later refined in a paper by the first author, \cite{Plociniczak2022} in which all assumptions were substantially weakened and the result was generalized to the semilinear case, building a passage to the quasilinear problem (which was analyzed in \cite{Plociniczak2023} in the case of regular initial data). However, in all of these papers, the devised numerical methods were semi-discrete. This paper also includes time discretization and gives an error analysis. Obtained results agree with optimal error bounds for constant coefficient linear subdiffusion, namely: globally in time and locally in space for nonsmooth initial data. 

In what follows we will assume the following regularity of the source and diffusivity
\begin{equation}\label{eq:contraints}
\begin{gathered}
    |f(x,t,u)|\leq F,\quad |f(x,t,u)-f(x,t,v)|\leq L|u-v|,\\
    0<D_-\leq D(x,t)\leq D_+,\quad |\partial^1_t D(x,t)|\leq \kappa(t)\in L^1((0,T)), \quad \nabla D(\cdot,t)\in L^2(\Omega)
\end{gathered}
\end{equation}
where $D_\pm$, $F$ and $L$ are positive constants. These assumptions are the same as in \cite{Plociniczak2022} and can be thought of as natural for the non-degenerate subdiffusion. Notice that for the time-dependence of the diffusivity we assume only absolute continuity, which is much weaker than the assumptions made in previous papers \cite{mustapha2018fem,jin2019subdiffusion} (see the discussion in \cite{Plociniczak2022}). The requirement of boundedness of the source $f$ can be relaxed by standard arguments presented, for example, in \cite{thomee2007}, Chapter 13. For simplicity, we will use the stringent regularity assumption above to make the presentation clear. 

We will use the notation $D(t) = D(\cdot, t)$ and $f(t,u) = f(\cdot, t, u)$ in order to clearly indicate the time-depenedence of the coefficient and the source. By standard integration by parts argument, the weak form of \eqref{eq:maineq} is
\begin{equation}\label{eq:weakeq}
\left(\partial^\alpha_tu,\chi\right)+a\left(D(t);u,\chi\right)=\left( f(t,u),\chi \right), \quad \chi\in H_0^1,
\end{equation}
where $(\cdot,\cdot)$ is the $L^2$ inner product, and 
\begin{equation}
    a(w;u,v)=\int\limits_\Omega w\; \nabla u\cdot \nabla v \,dx,
\end{equation}
is a bilinear form in $u$ and $v$. Now, let $u_h$ be the spatially discretized solution fulfilling
\begin{equation}
\left(\partial^\alpha_tu_h,v_h\right)+a\left(D(t);u_h,v_h\right)=\left( f(t,u),v_h \right), \quad v_h\in V_h,
\end{equation}
where $u_h(x,0)=P_h\ph(x)$ is the orthogonally projected initial condition into the space $V_h\subset H_0^1$ of piecewise linear continous functions over $\Omega_h\subset\overline{\Omega}$, that is, the quasiuniform triangulation of $\Omega$ with maximal diameter $h$. Of course, $\dim V_h < \infty$. Next, introduce the time grid $t_n = n\Delta t$ with $\Delta t=T/N$ and use the $L1$ scheme for the Caputo derivative denoted $\delta_t^\alpha$ defined as \cite{Plociniczak2023} 
\begin{equation}\label{eq:caputoscheme}
    \delta_t^\alpha y(t_n)=\frac{(\Delta t)^{-\alpha}}{\Gamma(2-\alpha)}\sum_{i=0}^{n-1}b_{n-i}(1-\alpha)(y(t_{i+1})-y(t_i)),
\end{equation}
with coefficients
\begin{equation}
    b_j(\beta)=j^\beta-(j-1)^\beta.
\end{equation}
In this way we can obtain a fully discrete scheme for finding the numerical approximation $u^n_h$ of the solution to our main equation \eqref{eq:maineq}
\begin{equation}\label{eq:fulldiscrete}
\left(\delta^\alpha_tu_h^n,v_h\right)+a\left(D(t_n);u_h^n,v_h\right)=\left( f(t_n, u^n_h),v_h \right).
\end{equation}
This equation gives rise to nonlinearities in the algebraic systems it produces due to the nonlinear source function, thus we will finish the scheme by extrapolating the values in that term. Therefore, we substitute $u_h^n$ in the RHS of \ref{eq:fulldiscrete} by $\hat{u}_h^n=2u_h^{n-1}-u_h^{n-2}$. The final scheme now reads
\begin{equation}\label{eq:finalscheme}
\left(\delta^\alpha_tu_h^n,v_h\right)+a\left(D(t_n);u_h^n,v_h\right)=\left( f(t_n, \hat{u}^n_h),v_h \right).
\end{equation}
As the scheme uses two past iterations, the first step must be done by solving the nonlinear variant given by \eqref{eq:fulldiscrete}. In the next section we prove the stability and convergence of the above numerical scheme along with theoretical estimates on orders of convergence. In Section 3 we present some numerical experiments that verify and illustrate our findings.

In what follows we adopt a notational convention to denote by $C>0$ a general constant dependent on $\alpha$, $T$, $\Omega$, $f$, and the exact solution $u$ or its derivatives. The generic constant $C$ does not depend on the numerical grid.
\section{Stability and convergence}
We will be interested in the $L^2$ norm of the error $u-u^n_h$, which can be decomposed as
\begin{equation}\label{eqn:Splitting}
    \|u-u_h+u_h-u^n_h\|\leq \|u-u_h\|+\|u_h-u^n_h\|.
\end{equation}
Assuming \eqref{eq:contraints} we can use Theorem 2 from \cite{Plociniczak2022} which gives us a bound on the error in space
\begin{equation}
    \|u-u_h\|\leq \|u-u_h\|+h\|\nabla(u-u_h)\|\leq Ch^2t^{-\frac{\alpha(2-p)}{2}}\|\ph\|_p,
\end{equation}
where $p\in[0,2]$. Thus, we can only focus on the error in time which we denote by $e_h^n=u_h-u_h^n$. First, we will prove the stability of the fully discrete scheme. 

\begin{thm}[Stability]
Let $u$ be the solution to \eqref{eq:finalscheme} and assume \eqref{eq:contraints}. Then,
\begin{equation}
    \|u^n_h\|\leq C\left(\|u_h^0\|+F\right)
\end{equation}
where the positive constant $C=C(\alpha,T,\Omega)$.
\end{thm}

\begin{proof}
Put $v_h=u_h^n$ in \eqref{eq:finalscheme} to obtain
\begin{equation}
\left(\delta^\alpha_tu_h^n,u_h^n\right)=-a\left(D(t_n);u_h^n,u_h^n\right)+\left(f(t_n, \hat{u}^n_h), u_h^n \right).
\end{equation}
Now utilizing the $(\delta_t^\alpha y^n,y^n)\geq \frac{1}{2}\delta_t^\alpha\|y^n\|^2$ inequality for $L^2$ functions (for example see Proposition 2 in \cite{Plociniczak2023}), noticing that $a$ is a positive-definite form and using the Cauchy-Schwarz inequality we get
\begin{equation}
    \frac{1}{2}\delta_t^\alpha\|u^n_h\|^2\leq \|f(t_n, \hat{u}^n_h)\|\|u_h^n\|.
\end{equation}
Utilizing the Cauchy inequality $ab\leq a^2+b^2$ and the assumption \eqref{eq:contraints} on $f$ resulting in a bound $\|f(t_n, u^n_h)\|\leq F|\Omega|$ we now get
\begin{equation}
    \delta_t^\alpha\|u^n_h\|^2\leq 2CF^2+2\|u_h^n\|^2.
\end{equation}
Now we can use the discrete fractional Gr\"{o}nwall inequality (see \cite{Liao2018}, Lemma 2.2), obtaining
\begin{equation}
    \|u^n_h\|^2\leq C\left(\|u_h^0\|^2+F^2\right).
\end{equation}
Using the inequality $a^2+b^2\leq(a+b)^2$ for non-negative real numbers we get the final result
\begin{equation}
    \|u^n_h\|\leq C\left(\|u_h^0\|+F\right),
\end{equation}
and the proof is complete. 
\end{proof}
Convergence can be proved with a Gr\"onwall inequality along with an estimate for the truncation error for the L1 discretisation of the Caputo derivative for solution with the following regularity estimate
\begin{equation}\label{eq:Regularity}
    \|\partial^m_t u(t)\|\leq C\left(1+t^{\alpha-m}\right), \quad m\in\{ 0,1,2 \}.
\end{equation}
This behaviour for small times is exactly what solutions to the linear subdiffusion equation exhibit for a general source term \cite{sakamoto2011initial}. A corresponding result has also been proved for semilinear equations with constant diffusivity in \cite{al2019numerical}. Existence, uniqueness, and H\"older regularity of solutions to a very general nonlinear subdiffusion problem has been investigated in \cite{topp2017existence} in the context of viscosity solutions. 

Before we proceed to the proof of convergence, we also need a bound for extrapolation error under our regularity assumptions.
\begin{lem}[Extrapolation error]\label{lem:Extrapolation}
    Let $y=y(t)$ be a function of time that satisfies \eqref{eq:Regularity}. Then the error of extrapolation $\hat{y}(t_n)=2y(t_{n-1})-y(t_{n-2})$ for $n\geq 2$ is bounded as
    \begin{equation}
        |\hat{y}(t_n)-y(t_n)|\leq
            C\left(1+t_n^{\alpha-1}\right)\Delta t,
    \end{equation}
    with a positive constant $C=C(y,\alpha)$.
\end{lem}
\begin{proof}
Let us start with the analysis of the error term
\begin{equation}
    |\hat{y}(t_n)-y(t_n)|=|y(t_n)-y(t_{n-1})|+|y(t_{n-1})-y(t_{n-2})|.
\end{equation}
Further, by writing the difference as the integral of a derivative, we have
\begin{equation}
    |\hat{y}(t_n)-y(t_n)|\leq \abs{ \,\int\limits_{t_{n-1}}^{t_n} y^{(1)}(t)dt}+\abs{ \,\int\limits_{t_{n-2}}^{t_{n-1}} y^{(1)}(t)dt}\leq  \int\limits_{t_{n-1}}^{t_n}\abs{ y^{(1)}(t)}dt+ \int\limits_{t_{n-2}}^{t_{n-1}} \abs{y^{(1)}(t)}dt.
\end{equation}
By the assumed regularity of $y$, we can now bound the derivatives, as
\begin{equation}
\begin{aligned}
    |\hat{y}(t_n)-y(t_n)|&\leq\int\limits_{t_{n-1}}^{t_n}C\left(1+t^{\alpha-1}\right)dt+ \int\limits_{t_{n-2}}^{t_{n-1}} C\left(1+t^{\alpha-1}\right)dt\\
    &=2C\Delta t+\frac{C}{\alpha}\left( 
 t_{n}^\alpha-t^\alpha_{n-2}\right)=2C\Delta t+\frac{C}{\alpha}(\Delta t)^\alpha\left( 
 n^\alpha-(n-2)^\alpha\right).
\end{aligned}
\end{equation}
The function $n\mapsto n^\alpha-(n-2)^\alpha$ is decreasing, bounded, and asymptotically equivalent to $2\alpha n^{\alpha-1}$ as $n\rightarrow\infty$. Therefore, there exists a constant $C$ such that
\begin{equation}
|\hat{y}(t_n)-y(t_n)|\leq
            C\left(\Delta t+(\Delta t)^\alpha n^{\alpha-1}\right),
\end{equation}
which concludes the proof. 
\end{proof}
\begin{thm}[Convergence]\label{thm:Convergence}
Let $u$ be the solution to \eqref{eq:finalscheme} and assume \eqref{eq:contraints} together with \eqref{eq:Regularity}. Then, for sufficiently small $\Delta t$ we have
\begin{equation}
    \|e^n_h\|\leq C\left(h^{r}\|\ph\|_r+(\Delta t)^{\alpha}\right),
\end{equation}
where the positive constant $C=C(\alpha, T, \Omega, u)$.
\end{thm}

\begin{proof}
Since $e_h^n \in V_h$ we can plug it into the left-hand side of \eqref{eq:finalscheme} to obtain
\begin{equation}
\begin{aligned}
\left(\delta^\alpha_te_h^n,v_h\right)+a\left(D(t_n);e_h^n,v_h\right)&=\left(\delta^\alpha_tu_h,v_h\right)+a\left(D(t_n);u_h,v_h\right)-\left(\delta^\alpha_tu_h^n,v_h\right)-a\left(D(t_n);u_h^n,v_h\right)\\
&=\left(\delta^\alpha_tu_h,v_h\right)+a\left(D(t_n);u_h,v_h\right)-\left( f(t_n,\hat{u}_h^n),v_h \right)\\
&=\left(\delta^\alpha_tu_h,v_h\right)-\left(\partial^\alpha_tu_h,v_h\right)+\left( f(t_n,u_h),v_h \right)-\left( f(t_n,\hat{u}_h^n),v_h \right)\\
&=\left(\delta^\alpha_tu_h-\partial^\alpha_tu_h,v_h\right)+\left( f(t_n,u_h)-f(t_n,\hat{u}_h^n),v_h \right).
\end{aligned}
\end{equation}
Now choosing $v_h=e_h^n$, utilizing the $(\delta_t^\alpha y^n,y^n)\geq \frac{1}{2}\delta_t^\alpha\|y^n\|^2$ inequality for the $L^2$ functions (see Proposition 2 in \cite{Plociniczak2023}), noticing that $a$ is a positive-definite form, and using the Cauchy-Schwarz inequality, we get
\begin{equation}
    \frac{1}{2}\delta_t^\alpha\|e^n_h\|^2\leq \left( \|\delta_t^\alpha u_h-\partial_t^\alpha u_h\|+\|f(t_n,u_h)-f(t_n,\hat{u}_h^n)\| \right)\|e_h^n\|.
\end{equation}
Since $f$ is a Lipschitz function in the $u$ variable, we further have
\begin{equation}
\frac{1}{2}\delta_t^\alpha\|e^n_h\|^2\leq \|\delta_t^\alpha u_h-\partial_t^\alpha u_h\|\|e_h^n\|+C\|e_h^n\|\|u_h-\hat{u}_h^n\|.
\end{equation}
We can split the error of the temporal approximation combined with the extrapolation method as
\begin{equation}
    \|u_h-\hat{u}_h^n\|=\|u_h-u_h^n+u_h^n-\hat{u}_h^n\|\leq \|u_h-u_h^n\|+\|u_h^n-\hat{u}_h^n\|\leq \|e^n_h\|+C(\Delta t)^\alpha,
\end{equation}
where we used the upper bound $(\Delta t)^\alpha$ for the the extrapolation error from Lemma \ref{lem:Extrapolation}. We now get
\begin{equation}
\frac{1}{2}\delta_t^\alpha\|e^n_h\|^2\leq \|\delta_t^\alpha u_h-\partial_t^\alpha u_h\|\|e_h^n\|+C\left(\|e_h^n\|^2+\|e_h^n\|(\Delta t)^{\alpha}\right).
\end{equation}
Using the Cauchy inequality $ab\leq \frac{1}{2}(a^2+b^2)$ we arrive at
\begin{equation}
    \delta_t^\alpha\|e^n_h\|^2\leq \|\delta_t^\alpha u_h-\partial_t^\alpha u_h\|^2+C\left(\|e_h^n\|^2+(\Delta t)^{2\alpha}\right).
\end{equation}
We need to bound the discretization error of the L1 scheme provided the assumed time regularity \eqref{eq:Regularity}. The truncation error for the L1 scheme is the following (see \cite{Stynes2016}, Lemma 5.2)
\begin{equation}
    \|\delta_t^\alpha u_h-\partial_t^\alpha u_h\|\leq C n^{-\alpha},
\end{equation}
where the constant $C$ can now depend on $u$ and its derivatives. We then have
\begin{equation}
\delta_t^\alpha\|e^n_h\|^2\leq C\left(n^{-2\alpha}+\|e_h^n\|^2+(\Delta t)^{2\alpha}\right).
\end{equation}
Now we can utilize the discrete fractional Gr\"{o}nwall inequality (see \cite{Liao2018}, Lemma 2.2) to conclude that
\begin{equation}
    \|e^n_h\|^2\leq C\left(\|e^0_h\|^2+2(\Delta t)^{2\alpha}\right).
\end{equation}
Due to the known fact concerning orthogonal projection, we can estimate the initial error $\|e^0_h\|=\|P_h\ph-\ph\|\leq Ch^{p}\|\ph\|_p$, $p\in (0,2]$ for $\ph\in H^{p}_0$ (see \cite{thomee2007}). Using that bound along with the elementary inequality $a^2+b^2\leq(a+b)^2$ we finally obtain 
\begin{equation}
    \|e^n_h\|\leq C\left(h^{p}\|\ph\|_p+(\Delta t)^{\alpha}\right),
\end{equation}
and the proof is complete. 
\end{proof}

Combining the above with the error decomposition \eqref{eqn:Splitting} we obtain the final error estimate for the fully discrete scheme.
\begin{cor}\label{cor:Convergence}
Let the assumptions of Theorem \ref{thm:Convergence} be satisfied. Then if $u$ is the solution to the PDE \eqref{eq:maineq} while $u^n_h$ its numerical approximation computed from \eqref{eq:finalscheme}, we have
\begin{equation}
    \|u-u_h^n\|\leq C\left((\Delta t)^{\alpha}+h^{p}\|\ph\|_p+h^2t_n^{-\frac{\alpha(2-p)}{2}}\|\ph\|_p\right), \quad p\in(0,2],
\end{equation}
where $C=C(\alpha,T,\Omega, u)$. 
\end{cor}

\section{Numerical examples}
For clarity of the presentation, we will discuss numerical experiments in 1 spatial dimension, that is when $\Omega=[0,1]$. As a basis for $V_h$ we choose a space of $M-2$ shifted and scaled tent functions, with elements defined for $i\in\{ 2,\ldots,M-1 \}$ as
\begin{equation}
     \psi_{h,i}(x)=\begin{cases}
        \dfrac{x-x_{i-1}}{h}, \quad x\in [x_{i-1},x_i],\\
        \dfrac{x_{i+1}-x}{h}, \quad x\in [x_{i},x_{i+1}],\\
        0,\quad \text{otherwise}
    \end{cases},
\end{equation}
where $x_i=(i-1)h$, $i\in\{1,\ldots,M\}$, $x_M=1$. We now expand $u_h^n$ in this basis to obtain
\begin{equation}
    u_h^n=\sum\limits_{i=2}^{M-1}y_i^n\psi_{h,i}(x).
\end{equation}
Furthermore, we denote $\bm{y}^n=(y_2^n,\ldots,y_{M-1}^n)$ and plug $v_h=\psi_{h,i}$ into the Galerkin scheme \eqref{eq:finalscheme} to obtain an algebraic system
\begin{equation}
    \bm{S}\delta^\alpha_t\bm{y}^n+\bm{A}^n\bm{y}^n=\bm{f}^n(\bm{\hat{y}}^n).
\end{equation}
The mass matrix $\bm{S}=\{S_{ij}\}_{i,j=2}^{M-1}$, the stiffness matrix $\bm{A}^n=\{A_{ij}^n\}_{i,j=2}^{M-1}$, and the load vector $\bm{f}^n=\{f_{i}^n\}_{i=2}^{M-1}$ are defined by
\begin{equation}
    S_{ij}=(\psi_{h,i},\psi_{h,j}),\quad
    A_{ij}^n=a\left(D(t_n);\psi_{h,i},\psi_{h,j}\right),\quad
    f^n_i=\left(f\left(t_n,\sum\limits_{i=2}^{M-1}\hat{y_i^n}\psi_{h,i}(x)\right),\psi_{h,i}\right).
\end{equation}
Applying the L1 scheme to the Caputo derivative \eqref{eq:caputoscheme} leads us to
\begin{equation}
\left(\bm{S}\frac{h^{-\alpha}}{\Gamma{(2-\alpha)}} +\bm{A}^n \right) \bm{y}^n=b_{n}\bm{S}\bm{y}^0+\bm{S}\sum_{i=1}^{n-1}\left(b_{n-i}(1-\alpha)-b_{n-i+1}(1-\alpha)\right)\bm{y}^i+ \bm{f}^n(\bm{\hat{y}}^n),
\end{equation}
or
\begin{equation}
\bm{y}^n=\left(\bm{S}\frac{h^{-\alpha}}{\Gamma{(2-\alpha)}} +\bm{A}^n \right)^{-1} \left(b_{n}\bm{S}\bm{y}^0+\bm{S}\sum_{i=1}^{n-1}\left(b_{n-i}(1-\alpha)-b_{n-i+1}(1-\alpha)\right)\bm{y}^i+ \bm{f}^n(\bm{\hat{y}}^n)\right).
\end{equation}
Below, for all numerical experiments, we will consider $\alpha\in\{\frac{1}{3},\frac{1}{2},\frac{2}{3}\}$ and the diffusivity defined as
\begin{equation}
    D(x,t)=\frac{1}{10}\left(1+t^{\frac{1}{2}}\left( x(1-x) \right)^{\frac{2}{3}}\right).
\end{equation}
It is easy to notice that it satisfies our assumptions \eqref{eq:contraints}, as $D_t(x,\cdot)$ lies in $L^1((0,T))$ and $D_x(\cdot,t)$ lies in $L^2((0,1))$. We will consider two examples:
\begin{enumerate}
    \item an engineered solution with smooth initial condition and $u(x,t)=(1+t^\alpha)\ph(x)$,
    \item a realistic problem, being the  fractional Zeldovich–Frank-Kamenetskii (ZFK) equation with a non-smooth initial condition.
\end{enumerate}
The error will be global, taking the maximal error for all $t\in[0,1]$.

\paragraph{Engineered solution}\phantom{a}\\
Let $u(x,0)=\ph(x)=\sin{(\pi x)}$ and choose the source function in a way that $u(x,t)=\left(  1+t^\alpha\right)\ph(x)$ is the solution to \eqref{eq:maineq}, that is
\begin{equation}
    f(x,t,u)=\left(\frac{\Gamma{(\alpha+1)}}{(1+t^{\alpha})}+\pi^2D(x,t)\right)u-\pi(1+t^\alpha)\cos{(\pi x)}\frac{\partial}{\partial x} D(x,t).
\end{equation}
According to Corollary \ref{cor:Convergence} the spatial order of convergence is $2$, as the initial condition lies in $H^2_0$, while the temporal one in the worst case is equal to $\alpha$. To choose an appropriate number of spatial degrees of freedom, that is $M$, we have to ascertain that the error in time dominates the error in space, that is,
\begin{equation}
    N^{-\alpha}\gg M^{-2}\implies M\gg N^{\frac{\alpha}{2}}.
\end{equation}
For example, choosing $N=2^{10}$ and $\alpha=2/3$ we get $M\gg 2^{\frac{10}{3}}>2^4$. The choice of $M=2^6$ thus seems to be well suited for the problem considered. We can see the results for the global error in time on the log-log scale in Fig. \ref{fig:nonsmooth_temporal}. The error behaves as predicted by the theory, even for small values of $N$, that is, the order of convergence agrees with $\alpha$. Furthermore, the local spatial error computed at $t=1$ is presented in a log-log scale in Fig. \ref{fig:nonsmooth_spatial}. As can be seen, the span of the $M$-parameter is relatively small due to fast convergence. As can be seen, numerical computations regarding the engineered solution verify the assertion of Corollary \ref{cor:Convergence}. 

\begin{figure}
    \centering
    \includegraphics[scale=0.35]{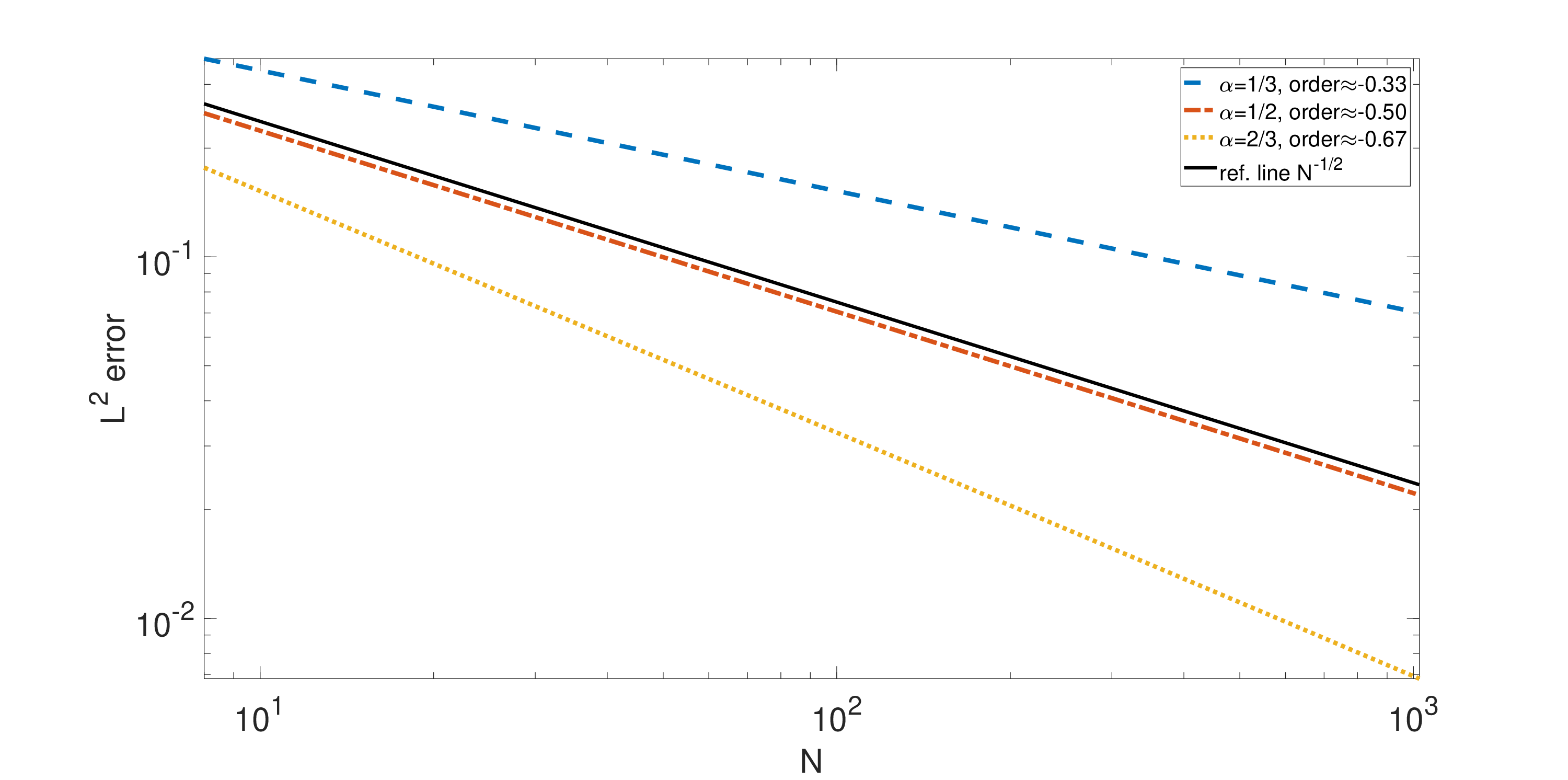}
    \caption{A log-log plot of the global in time $L^2$ error as a function of the number of temporal gridpoints $N$ with the number of spatial degrees of freedom $M=2^6$. }
    \label{fig:nonsmooth_temporal}
\end{figure}

\begin{figure}
    \centering
    \includegraphics[scale=0.35]{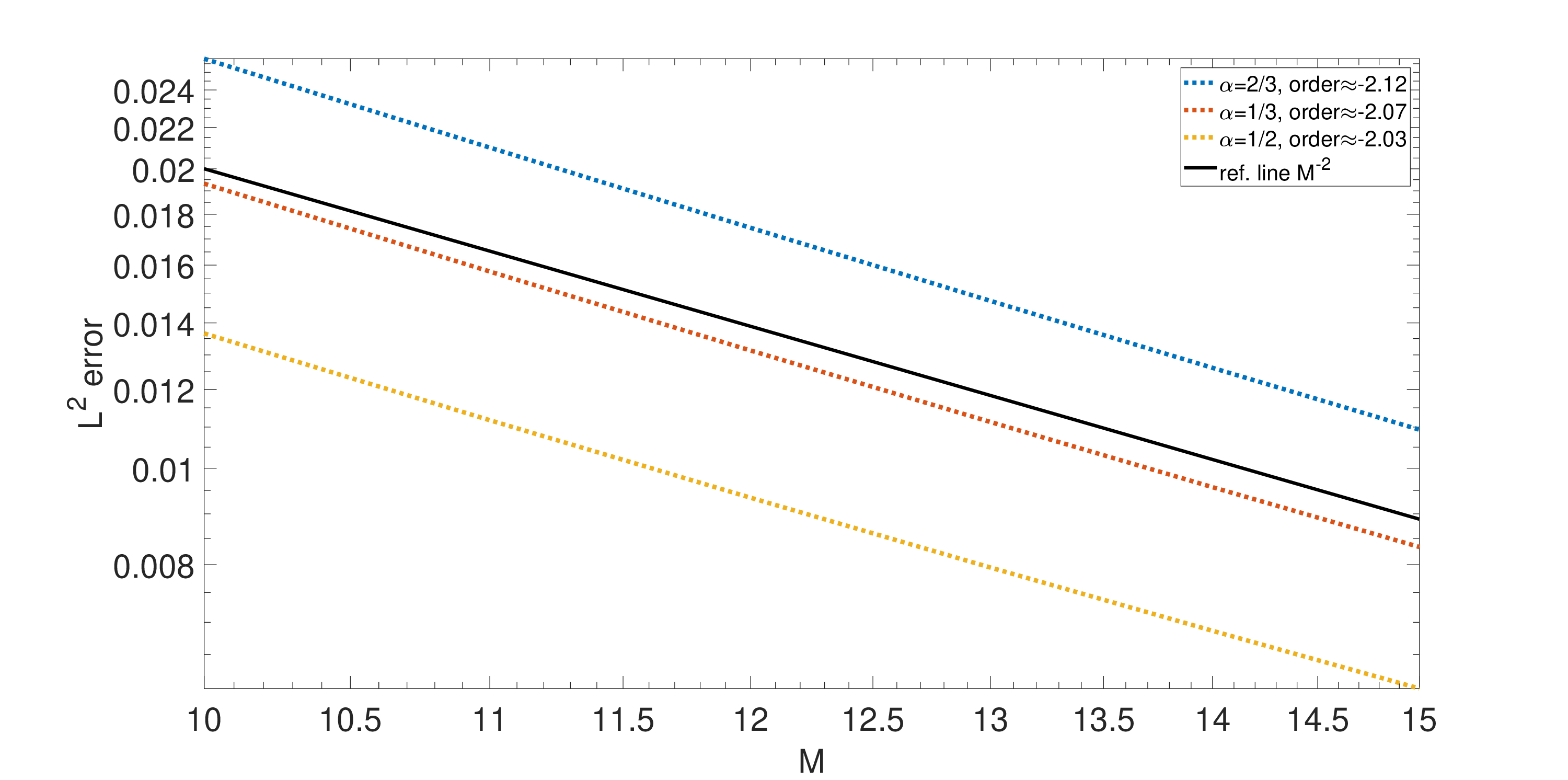}
    \caption{A log-log plot of the local in time $L^2$ error as a function of the number of spatial degrees of freedom $M$ with the number of temporal gridpoints $N=10^3$. }
    \label{fig:nonsmooth_spatial}
\end{figure}

\begin{figure}
    \centering
    \includegraphics[scale=0.35]{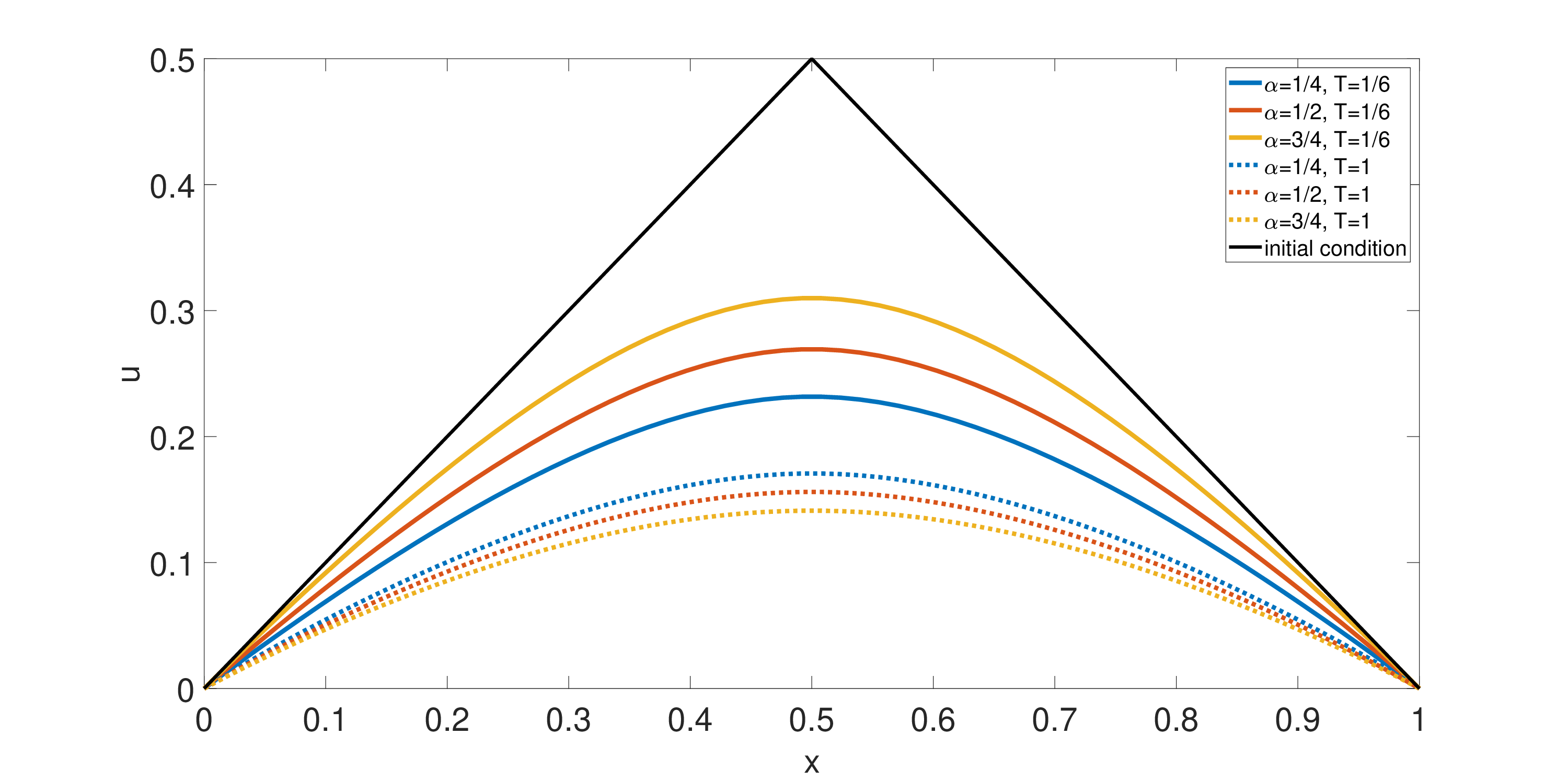}
    \caption{Exemplary solutions of the time-fractional ZFK equation with variable diffusivity. The solutions are obtained separately, with $N=2^{12}$ and $M=2^6$.}
    \label{fig:ZFK}
\end{figure}

\paragraph{ZFK equation}\phantom{a}\\
The second example is the subdiffusive Zeldovich–Frank-Kamenetskii (ZFK) equation with the time-dependent Zeldovich number $\beta$ and variable diffusivity, given by
\begin{equation}\label{eq:ZFK}
    \partial^\alpha_tu(x,t)=\frac{\partial}{\partial x}\left( D(x,t)\frac{\partial}{\partial x} u\right)+\frac{\beta^2(t)}{2}u(1-u)e^{\beta(t)(1-u)}.
\end{equation}
The classical equation models premixed flame propagation, where $u$ usually denotes the temperature in dimensionless units. The introduction of the nonlocality in time and variable diffusivity can be thought of as a generalization based on the properties of the medium in which the flame propagates. The time dependence of the Zeldovich number can model the change in temperature of the unburnt mixture, as
\begin{equation}
 \beta=\frac{E_a}{RT_b}\frac{T_b-T_u}{T_b},   
\end{equation}
where $E_a$ is the activation energy of the reaction, $R$ is the universal gas constant, $T_b$ is the burnt gas temperature and $T_u$ is the unburnt mixture temperature. For typical combustion, its value usually lies between $8$ and $20$, we choose $\beta$ to be
\begin{equation}
    \beta(t)=14+6\sin{(5\pi t)}.
\end{equation}
As the initial condition, we shall consider a function in $H_0^1$, but not in $H_0^2$, namely,
\begin{equation}
    \ph(x)=\frac{1}{2}-\abs{x-\frac{1}{2}}.
\end{equation}
Estimated solutions are shown in Fig. \ref{fig:ZFK}. We can see the expected behavior for anomalous diffusion, which is a faster mass movement for smaller times and small $\alpha$, while a faster mass movement at larger times for greater $\alpha$. The unboundedness of the time derivative close to $t=0$ can also be observed, as the solution changes very quickly compared to the later behavior. In Fig. \ref{fig:ZFK_derivative} we can see further examination of this observation. For example, the time derivative at a point $x=1/2$ blows up, behaving as a power function for $\alpha=3/4$, while for $\alpha=1/2,1/4$ the shape of the plot on a log-log scale suggests at least a power-like divergence. For estimating the spatial order of convergence, we will use Aitken's formula based on extrapolation
\begin{equation}
    order^x\approx \log_2{\frac{\| u_{\frac{h}{2}}^N-u_h^N\|}{\| u_{\frac{h}{4}}^N-u_{\frac{h}{2}}^N\|}}.
\end{equation}
The order of convergence in time will be estimated analogously, with varying $N$, $\forall_{k}\,kN=T$, and constant $T=1$
\begin{equation}
    order^t\approx \log_2{\frac{\| u_h^{2N}-u_h^N\|}{\| u_h^{4N}-u_h^{2N}\|}}.
\end{equation}
We also consider a global temporal error order given by
\begin{equation}
    order_\infty^t\approx \log_2{\frac{\max\limits_{n\in \{1,\ldots,N\}}\| u_h^n-u_h^{2n}\|}{\max\limits_{n\in \{1,\ldots,N\}}\| u_h^{4n}-u_h^{2n}\|}},
\end{equation}
where the solutions $u_h^{kn}$ are defined for a temporal grid with $kN$ points. We can see the results for the temporal convergence in Tab. \ref{tbl:ZFK_temporal}. The spatial convergence predicted by the theory is of order $1$, while the global temporal one of order $\alpha$ - the choice of $M=2^9$, $N=2^6$ is thus well suited, as the local error is of order $1$. The results show that the global order of convergence is close to the expected value of $\alpha$ and increases with increasing $\alpha$. The estimated order is somewhat smaller than anticipated - this could be the result of error of the method of extrapolation, as well as additional slow convergence to the spatial error.
\begin{figure}
    \centering
    \includegraphics[scale=0.35]{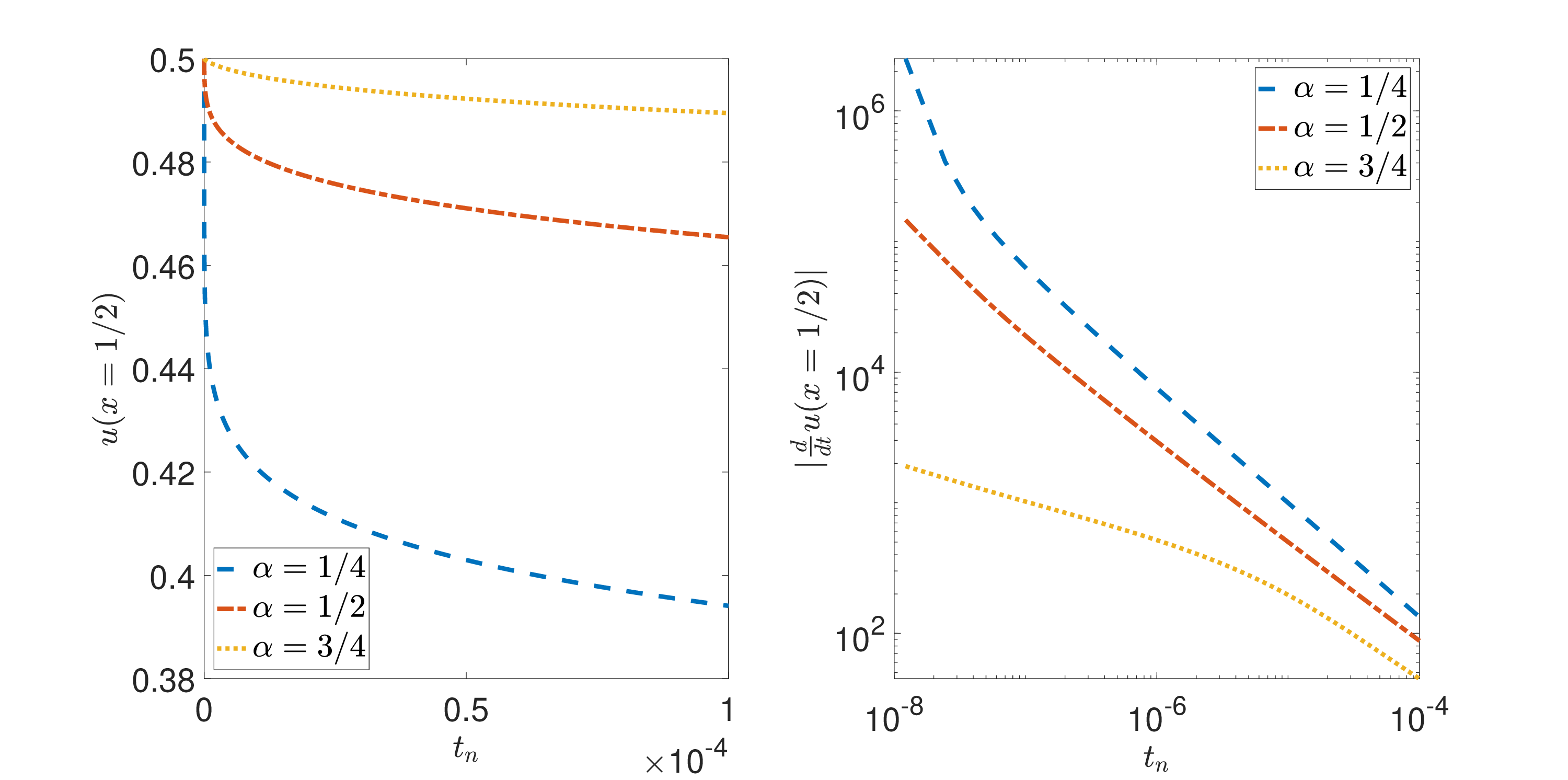}
    \caption{On the left: solutions of \eqref{eq:ZFK} at a point $x=1/2$, with $T=10^{-4}$, $N=2^{13}$, $M=2^6+1$. On the right: timer derivatives on a log-log scale.}
    \label{fig:ZFK_derivative}
\end{figure}

\begin{table} 
    \centering
    \begin{tabular}{cccc}
        \toprule
        norm\textbackslash $\alpha$ & $1/3$ & $1/2$ & $2/3$\\
        \midrule
        $order^t$ & 1.03 & 1.04 & 1.05 \\
        $order_\infty^t$ & 0.23 & 0.38 & 0.5 \\
        \bottomrule
    \end{tabular}
    \caption{Results for the order of temporal convergence, for $M=2^9$ and $N=2^6$. }
    \label{tbl:ZFK_temporal}
\end{table}

In Tab. \ref{tbl:ZFK_spatial} we can see the result for spatial convergence. As predicted, orders of convergence agree with the value $1$. The choice of $N=2^{11}$ makes the temporal error negligible compared to the spatial error for $M=2^6$. The results show that the order is indeed close to $1$ and independent of $\alpha$, which was expected due to the initial condition lying in $H^1_0$, but not $H^2_0$.
\begin{table} 
    \centering
    \begin{tabular}{cccc}
        \toprule
        $\alpha$ & $1/3$ & $1/2$ & $2/3$\\
        \midrule
         $order^x$ & 1.05 & 1.05 & 1.05 \\
        \bottomrule
    \end{tabular}
    \caption{Results for the order of spatial convergence, for $N=2^{11}$ and $M=2^4$.}
    \label{tbl:ZFK_spatial}
\end{table}

\section*{Acknowledgements}
This research was supported by \textit{Narodowe Centrum Nauki} (NCN) Sonata Bis grant with number 2020/38/E/ST1/00153.

\printbibliography
\end{document}